\documentclass[12pt,a4paper]{article}

\usepackage[utf8]{inputenc}
\usepackage{amsfonts, amsmath, amssymb, amsthm, mathrsfs}

\usepackage{cite}

\usepackage{mathtools}
\usepackage{geometry, graphicx}
\usepackage[hidelinks]{hyperref}
\usepackage{cleveref}
\usepackage{enumitem}
\hypersetup{colorlinks={true},linkcolor={blue}}

\usepackage{lipsum}
\usepackage{tikz-cd}
\usepackage{verbatim}
\usepackage{bm}

\usepackage{setspace,caption}
\usepackage{algorithm}
\usepackage[noend]{algpseudocode}

\crefformat{equation}{(#2#1#3)}

\graphicspath{ {./images/} }

\DeclarePairedDelimiter\floor{\lfloor}{\rfloor}

\DeclarePairedDelimiterX{\norm}[1]{\lVert}{\rVert}{#1}

\newcommand\set[1]{\left\{ #1 \right\}}

\DeclareMathOperator{\E}{\mathbb{E}} 
\DeclareMathOperator{\Prb}{\mathbb{P}} 

\newtheorem{thm}{Theorem}

\newtheorem{defi}{Definition}
\newtheorem{lem}{Lemma}

\newtheorem{rmk}{Remark}

\begin{document}
\title{\textbf{A survey on combinatorial optimization}}
\author{\textbf{Phuong Le}}
%\date{}
\maketitle

\begin{abstract}
This survey revisits classical combinatorial optimization algorithms and extends them to two-stage stochastic models, particularly focusing on client-element problems. We reformulate these problems to optimize element selection under uncertainty and present two key sampling algorithms: \texttt{SSA} and \texttt{Boost-and-Sample}, highlighting their performance guarantees. Additionally, we explore correlation-robust optimization, introducing the concept of the correlation gap, which enables approximations using independent distributions with minimal accuracy loss. This survey analyzes and presents foundational combinatorial optimization methods for researchers at the intersection of this field and reinforcement learning.
\end{abstract}

\section{Introduction}
Combinatorial optimization is a crucial problem with applications across various domains. The recent advancements in machine learning, particularly in reinforcement learning (RL), have revealed strong connections between combinatorial optimization and RL, as suggested in \cite{RL_for_CO}. Notably, RL \cite{RL_for_CO_2} has been used to train neural networks to solve complex problems like the Traveling Salesman Problem (TSP). Although most combinatorial optimization problems are discrete, a recent paper by \cite{dpo} connects traditional discrete reinforcement learning with Hamiltonian-based continuous dynamical systems, further bridging the gap with continuous-time physics-based formulations.

Given the potential to advance combinatorial optimization and its integration with other fields, this survey revisits classical algorithms in this domain. The aim is to provide a foundation that future researchers can build upon to develop machine learning-based approaches. In this paper, we reformulate many classical combinatorial optimization problems as client-element problems: given two sets (clients and elements), the task is to select elements optimally to serve a subset of clients. We then define the two-stage stochastic version of these deterministic problems, where randomness arises from the set of clients, which will be defined shortly. In the first stage, elements are purchased in advance at a lower cost, and in the second stage, once the set of clients is realized, additional elements are acquired if necessary. The goal is to optimize the selection of elements in both stages.

To solve these stochastic versions, we rely on the (approximate) solutions of their deterministic counterparts using sampling methods. This survey presents two sampling algorithms: \texttt{Boost-and-Sample} (see \cite{gupta_pal_ravi_sinha_2004}) and its variant \texttt{Ind-Boost}, along with \texttt{SSA} (see \cite{swamy_shmoys}). We also provide brief proofs of their performance guarantees, showing that under certain assumptions, the outputs of these algorithms are within a constant factor of the optimal solution.

A challenge in sampling methods is that the distribution $\pi$ on the random subset of clients is often unknown and can be highly complex. Both algorithms require an efficient black-box method to draw samples from $\pi$, which may be difficult when the sample space is exponential in the number of clients, and $\pi$ is intricate. However, if $\pi$ is an independent distribution—where the need to serve each client is independent of the others—efficient (polynomial-time) sampling becomes possible.

In the final section, we introduce correlation-robust stochastic optimization and the concept of the correlation gap (see \cite{agrawal_ding_saberi_ye_2010}). For certain combinatorial optimization problems, we can work with a new independent distribution $\pi_1$. The ratio between the optimal solution for the original (or worst-case) distribution $\pi$ and that for the independent distribution $\pi_1$ is called the correlation gap. We will present a detailed theorem on the upper bound of the correlation gap for a specific class of two-stage stochastic combinatorial optimization problems. This theorem indicates that by working with the simpler independent distribution, we only lose a constant or small factor of accuracy.

\subsection{General formulation for a class of combinatorial optimization problems}
We start with the deterministic combinatorial optimization problem:
Let $V$ be the set of all clients, and $X$ be the set of elements that we can purchase. For each subset $F \subseteq X$, we define $c(F)$ to be the cost of the element set $F$. Thus, we have the cost function $c: 2^X \to \mathbb{R}_{\geq 0}$. Suppose that we have to serve a set of clients $S \subseteq V$. Let $\textbf{Sols}(S) \subseteq 2^X$ be the collection of all subsets $F$ (of elements in $X$) so that we can use elements in $F$ to serve our clients in $S$. We call each such $F$ a feasible solution. Then the problem is to find the optimal solution $F^{*}$ to serve $S$ with minimum cost. Let also denote by $\textbf{OPT}(S)$ this optimal solution $F^{*}$ with the minimum cost $c$ to serve the set of clients $S$. Now we will go through some classical combinatorial optimization problems that can be formulated in terms of the above problem:
\begin{enumerate}
	\item \textbf{Steiner tree problem}: Given a weighted graph $G = (V, E)$ with weight $w(e)$ on each edge $e \in E$ and a subset $S \subset V$ of vertices called terminal vertices. We have to find a tree that spans through the given set of terminal vertices with minimum total edge weights. Note that, if the set of terminal vertices consists of only two vertices, then this problem is equivalent to the \textbf{shortest path problem}. At the other extreme, if the set $S$ of terminal vertices are the whole $V$, then this problem is equivalent to the \textbf{minimum spanning tree problem}. Here the set of clients is $V$, and the set of elements to serve clients is the set of edges $X = E$. The cost $c(F)$ is simply $\sum_{e \in F} w(e)$. Serving a set of vertices means that finding a set of edges so that there is a tree (Steiner tree) spans all of these vertices with the edges being one of the edges in this set.

	\item \textbf{Uncapacitated Facility Location (UFL)}: In this problem, we have a set of clients (or cities or customers) $V$, and a set of (possible) facility sites $X$. Assume that we have to serve a subset of clients $S$ by using built facilities. In this problem, we have to choose a set $F \subseteq X$ of the facilities to build (open) so that the cost, which consists of the total distance from the clients in $S$ to the nearest facilities and the cost of building (opening) facilities chosen in $F$, is minimized. One can easily see that $V, X, F$ and $S$ are matched to the same quantities in the general formulation. Also, for a fixed subset $S$ of $V$, the cost of distance and building facilities may depend only on $F$ and $S$.  
	\item \textbf{Set cover problem} Again we start with some (universal) set (of clients) $V$. We are also given a collection $\mathcal{S}$ of subsets of $V$. We have to serve some random subset $S$ of $V$. Our element is a subcollection of subsets in $\mathcal{S}$, and the union of sets in this subcollection has to contain our random set $S$ (of clients).
	\item \textbf{Vertex cover} In this problem, a graph $G=(V, E)$, and the set of clients $S$ is some random set of edges. That set of elements that we can purchase is some subset of $F$ of vertices (Here $X = E$). To serve $S$, we need to choose sufficiently many vertices so that each edge in $S$ is incident to at least one vertex in $F$.	
	\item \textbf{Multicommodity flow} It is a bit tricky to formulate this problem in term of our general formulation. This problem belongs to a more general class than the one we considered. The random data here is not a subset of anything, but is the set of demands from source $s_i$ to the sink $t_i$. The elements here are the edges to be purchased in advanced together with the capacities we want to install on them. Purchasing edges later costs more, so we want to go ahead to purchase some edges, and try to optimize our decision of purchasing these edges. Nevertheless, this problem has a (stochastic) linear programming formulation and can be solved by SSA algorithm described below.
\end{enumerate}
\vspace{2mm}

\subsection{Two-stage stochastic optimization formulation}
Based on the previous deterministic formulation of the combinatorial optimization problems that we're interested in, we can define the two-stage stochastic version of these problems. Each problem has two stages, and the cost of purchasing elements in the first stage is lower than the cost in the second stage. We define the cost for purchasing each element $e \in X$ in the first and second stages to be $w^I_e$ and $w^{II}_e$ respectively. Thus, if we purchase a set $F$ in the first stage, then the cost is $c_1(F) = \sum_{e\in F}w^I_e$. Similarly, to purchase the set of elements $F$ in the second stage, the cost is $c_2(F) = \sum_{e\in F}w^{II}_e$.\\

More specifically, we define a possible solution for two-stage stochastic optimization as follows:
\begin{itemize}
	\item In the first stage, the set of clients $S \subset V$ is unknown, and we assume that the randomness is based on a distribution $\pi$ on the collection of subsets $2^V$ of $V$. We can access this distribution via a black box, and can draw a sample from it in $poly(|V|)$ (polynomial time in $|V|$). Suppose that we buy a set of elements $F_0 \subseteq X$ at this stage with cost $c^1(F_0)$.
	\item In the second stage, the random set $S\subset V$ is realized (according to the distribution $\pi$). We also assume that $S$ is conditionally independent of any of our actions in the first stage. Then the second-stage solution (also called recourse) will consist of the set of elements $F_S \subseteq X$ so that $F_0 \cup F_S$ can serve the realized set of client $S$ (Since $S$ is random, $F_0$ will not be enough to serve $S$ and we may need to add an additional subset $F_S$ of $X$ to serve $S$). 
\end{itemize}

It is also clear from section \textbf{1.2} how we can get the two-stage stochastic versions for classical combinatorial optimization problems listed in \textbf{1.2}.

\subsection{Linear Programming Formulation}
Let us try to formulate (some special cases of) the above two-stage stochastic problems in terms of a linear programming (LP) problem. Now we consider vectors $x \in \mathbb{R}^{|X|}$, and index such vectors  by element $e \in X$: $x = (x_e)_{e \in X}$. Let also use $x$ to denote our decision in stage 1, i.e, the indicator vector of the elements we purchase in the first stage: $x_e = 1$ iff we choose to purchase element $e$ in the first stage, and $x_e = 0$ otherwise. The cost in stage $1$ is then $w^I  \cdot x = \sum_e w_e^I x_e$. Now let $S$ be the random subset in stage $2$, and we let $p_A = \Prb(\textbf{this random set is }A)$. So $A$ is one of the many possible values that $S$ can receive, and we call $A$ a scenario. Then, by assuming that $f_A(x)$ is the additional cost to serve $A$ with the initial set of elements indicated $x$, our objective function becomes
$$\min \left(w^I \cdot x + \sum_{A \subset V}p_A f_A(x) \right), \text{where }x \in \set{0, 1}^{|X|}$$
We consider the convex relaxation of this problem and minimize on $x \in \mathcal{P}$ where $\mathcal{P}$ is the polytope $[0, 1]^{|X|}$. Often, the $f_A(x)$ can be found by solving the following linear programming problem with coefficients depending on $x$ and $A$:
$$f_A(x) = \min \{w^A\cdot r_A + q^A\cdot s_A: r_A \in \mathbb{R}^m_{\geq 0}, s_A \in \mathbb{R}^{n}_{\geq 0}, D^A s_A + T^A r_A \geq j^A - T^Ax \}$$
Here we consider a little bit more general formulation than the one in section \textbf{1. 3}: $w^A$ is the cost at the scenario $A$, and in this case, it can be the constant $w^{II}$, the cost in stage $2$ that we mentioned above, or it can be scenario-dependent.\\

Let us look at $f_A$ function of the stochastic version of the uncapacitated facility location (UFL) problem. So in stage $2$, suppose that $A$ is the set of clients that one needs to serve. Then let $r_A \in \mathbb{R}^{|X|}$ to be the indicator vector of the facilities chosen to be built: $(r_A)_e = 1$ if $e \in X$ is built in stage $2$ and $0$ otherwise. Also, let $s_A \in \mathbb{R}^{|X|\times |V|}$ so that $(s_A)_{ij} = 1$ if we assign client $j$ to facility $i$. Also, let $c_{ij}$ be the cost when facility $i$ serves client $j$: it can be the distance function (with respect to some metric) between $i$ and $j$. Then the cost function in scenario $A$ is the total cost of building facilities chosen, which is $w^A \cdot r_A$ and the cost of serving clients, which is $\sum_{ij}c_{ij} (s_A)_{ij}$ so that it can be written as $q^A\cdot s_A$ with $q_A = c$ that can depend on scenario $A$.  First, we require that we only assign some client $j$ to facility $i$ iff facility $i$ is already built. This is equivalent to $(s_A)_{ij} \leq x_i + (r_A)_i$ for all $j \in A$ and all $i \in X$. Furthermore, with the assignment $s_A$, every client in $A$ has to be served by some facility. This is equivalent to $\sum_i (s_A)_{ij} \geq 1$ for all $j \in A$. These two constraints is equivalent to $D^As_A + T^Ar_A \geq j^A - T^Ax$ for appropriate coefficient matrix $D^A, T^A$, and coefficient vector $j^A$ that can depend on $A$. Even though, in the original combinatorial problem, $r_A$, and $s_A$ are $\set{0, 1}$-vector, by convex relaxation, we can assume that the constraint set takes the form as in definition of $f_A$. There is a concept of integrality gap, which indicates how well the convex relaxation can do to approximate the true integer or $\set{0, 1}$ solution. We will not go into that and for the \textit{SSA} algorithm we can assume that relaxation problem approximates the original problem within a small constant factor. This is true for \textbf{UFL} problem.
\vspace{2mm}

\section{Two methods for solving stochastic problems based on solutions for deterministic problem}
\subsection{Boosted Sampling}
\subsubsection{ \texttt{Boost-and-Sample} algorithm}
We begin with the cost-sharing function definition. We will encounter a similar cost-sharing function later with the additional variable being an order on the client set.
\begin{defi}
Given an $\alpha$-approximation algorithm $\mathcal{A}$ for the deterministic problem $Det$, the function $\xi: 2^V\times V \to \mathbb{R}_{\geq 0}$ is a $\beta$-strict cost sharing function if the following properties hold:
\begin{enumerate}
	\item For a set $S \subseteq V, \xi(S, j) > 0$ only for $j \in S$
	\item For a set $S \subseteq V, \xi(S, S) = \sum_{j \in S} \xi(S, j) \leq c(\textbf{OPT}(S))$ (\textit(fairness))
	\item If $S' = S \cup T$, then $S(S', T) = \sum_{j \in T} \xi(S', j) \geq (1/\beta) \times$ cost of augmenting the solution $\mathcal{A}(S)$, which is the output of the algorithm $\mathcal{A}$ given the input $S$, to a solution in $\textbf{Sols}(S')$ (\textbf{strictness}). \\
\end{enumerate}
For the third property, we assume there exists a polynomial time algorithm \textbf{Aug}$_{\mathcal{A}}$ that depends on the algorithm $\mathcal{A}$ and can augment $\mathcal{A}(S)$ to a solution in $\textbf{Sols}(S')$\\
\end{defi}
\vspace{2mm}

From now on, we define $\xi(S, A)$ as the sum $\sum_{j \in A} \xi(S, j)$. In this section (two algorithms), we assume that the cost in second stage for every element $e \in X$ is a constant factor of that in the first stage: $w^{II}_e =\sigma w^I_e = \sigma c_e$ for each $e \in X$.  Now we consider the main algorithm:\\

\textbf{Algorithms} \texttt{Boost-and-Sample}
\begin{enumerate}
	\item Draw $\floor{\sigma}$ independent samples $D_1, D_2, \cdots, D_{\sigma}$ of the sample clients by sampling from the distribution $\pi$. Let $D = \bigcup_i D_i$.
	\item Using the algorithm $\mathcal{A}$, construct an $\alpha$-approximate first stage solution $F_0 \in \textbf{Sols}(D)$
	\item If the client set $S$ is realized in the second stage, use the augmenting algorithm \textbf{Aug}$_{\mathcal{A}}$ to compute $F_S$ such that $F_0 \cup F_S \in \textbf{Sols}(S)$. The output of the algorithm is $F_0 \cup F_S$
\end{enumerate}
\vspace{2mm}

Now we will state the theorem that guarantees the performance of this sampling approximate algorithm given this assumption.
\begin{thm}\label{thm1}
Consider a combinatorial optimization problem that is subadditive, and let $\mathcal{A}$ be an $\alpha$-approximation algorithm for its deterministic version $Det$ that admits a \textit{$\beta$-strict cost sharing function}. Then the \texttt{Boost-and-Sample} algorithm is an $(\alpha + \beta)$-approximation algorithm for the stochastic version of this combinatorial optimization problem.
\end{thm}
\begin{proof}
W.L.O.G assume that $\sigma \in \mathbb{Z}$. We will bound the expected costs of the first and second-stage solutions separately. Let $F_0^*$ be the first-stage component of the optimal solution, and $F_S^*$ be the second-stage component if the set of realized clients is $S$. Hence, the optimal cost $Z^* = c(F_0^*) + \sum_S \pi(S)(\sigma c(F_S^*))$. We also denote $Z_0* = c(F_0^*)$ and $Z_r^* = \sum_S \pi(S) \sigma c(F_S^*)$.\\

Let$F_1' = F_0^* \cup F_{D_1}^* \cup F_{D_2}^* \cup \cdots F_{D_{\sigma}}^*$. Then $F_1 \in \textbf{Sols}(D)$ by subadditivity of the combinatorial optimization problem. Then,
\begin{align*}
\E_D[c(F_1')] &\leq c(F_0^*) + \E_D[\sum_{i=1}^{\sigma} c(F_{D_i}^*)]= c(F_0^*) + \sum_{i = 1}^{\sigma} \E_{D_i}[c(F_{D_i}^*)]\\
&= c(F_0^*) + \sigma \sum_{S}\pi(S) c(F_S^*) = Z^*
\end{align*}
Because in step 2, the algorithm $\mathcal{A}$ construct the solution $F_0$ for the optimization problem on $\textbf{Sols}(D)$, $c(F_0) \leq \alpha c(F_1')$. Hence $\E_D[c(F_0)] \leq \alpha\E_D[c(F_1')] \leq \alpha Z^{*} $\\

Let $S$ be the set of realized clients, and let $F_S$ be the result of the algorithm $\textbf{Aug}_{\mathcal{A}}$ such that $F_0 \cup F_S \in \textbf{Sols}(S)$. We now will bound $\E[c(F_S)]$. By $\beta$-strictness of the cost sharing function $\xi$, $c(F_S) \leq \beta \xi(D \cup S, S\setminus D)$. \\

Consider another probabilistic-equivalent process to generate the set $D_i$ and $S$: Draw $\sigma+1$ independent samples $D'_1, D'_2, \cdots, D'_{\sigma+1}$ from the distribution $\pi$. Now choose a random value $t$ uniformly from $\set{1, 2, \cdots, \sigma+1}$, and set $S = D'_t$, and $D = \bigcup_{i \neq t} D'_i$. Because we are choosing the $\sigma+1$ sets independently, the process is identically distributed to the original process we used to produce $(D, S) =(D_1\cup \cdots \cup D_{\sigma}, S)$. Let $D' = \cup_{i=1}^{\sigma+1} D'_i$, and $D'_{-i} = \cup_{l \neq i} D'_l$ (union of all sets except $D'_i$). \\

By the fairness property of the cost sharing function, we have:
$$\sum_{i=1}^{\sigma+1} \xi(D', D'_i \setminus D'_{-i}) \leq c(\textbf{OPT}(D'))$$
Then,
$$\E_t[\xi(D', D'_i \setminus D'_{-i})] = \frac{1}{\sigma+1}\left(\sum_{i=1}^{\sigma+1} \xi(D', D'_i \setminus D'_{-i})\right) \leq \frac{1}{\sigma+1} c(\textbf{OPT}(D'))$$

Because the two process for picking $D'_i$ and $D_K$ is equivalent to the original process we used to pick $D$ and $S$, we have:
$$\E_{D, S}[\xi(D \cup S, S\setminus D)] = \E_{D',K}[\xi(D', D'_t\setminus D_t)] \leq {1}{\sigma+1} \E_{D'}[c(OPT(D'))]$$
The last inequality follows from the upper bound for $\E_t[\xi(D', D'_i \setminus D'_{-i})]$.\\

Now we will show that $\E_{D'}[c(OPT(D'))] \leq \frac{\sigma+1}{\sigma}Z^*$. Now we will use a similar argument used to bound expectation of first stage optimal solution $\E_D[c(F_0)]$. Let $F_2' = F_0^* \cup F_{D'_1}^* \cup F_{D'_2}^* \cup \cdots \cup F_{D'_{\sigma+1}}^*$. Again, because of the subadditivity of the combinatorial problem, $F_2' \in \textbf{Sols}(D')$. Then,
\begin{align*}
\E_{D'}[c(OPT(D'))] &\leq c(F_0^*) + \sum_{i=1}^{\sigma+1}\E_{D'_i}[c(F_{D'_i}^*)] \leq Z_0^* + (\sigma+1)\frac{Z_r^*}{\sigma} \\
&\leq \frac{\sigma+1}{\sigma}(Z_0^* + Z_r^*) = \frac{\sigma+1}{\sigma}Z^*
\end{align*}
Now combining all inequalities, we have:
\begin{align*}
\E_S[c(F_S)] &\leq \beta\E[\xi(D\cup S, S \setminus D)] \leq \beta\frac{1}{\sigma+1}\E_{D'}[c(OPT(D'))] \\
&\leq \beta\frac{1}{\sigma+1}\frac{\sigma+1}{\sigma}Z^* = \beta Z^{*}
\end{align*}

Thus, the expectation of our output is $\E[c(F_0)] + \E_S[c(F_S)]  \leq (\alpha + \beta)Z^*$ as desired.
\end{proof}
\vspace{2mm}

\subsubsection{ \texttt{Ind-Boost} algorithm}
In the special case $\pi$ is an independent distribution, i.e, each client $j \in V$ requires service with probability $\pi_j$ independently of all other clients. Then we also have a similar cost sharing function and a simpler algorithm \texttt{Ind-Boost}
\begin{defi}
Given an $\alpha$-approximation algorithm $\mathcal{A}$ for the deterministic problem $Det$, the function $\xi: 2^V \times V \to \mathbb{R}_{\geq 0}$ is a $\beta$-uni-strict cost sharing function if it satisfies the first two properties of $\beta$-strict cost sharing function and this third property:

If $S' = S \cup \set{j}$, then $\xi(S', j) \geq (1/\beta) \times$ cost of augmenting the solution $\mathcal{A}(S)$ to a solution in $\textbf{Sols}(S')$ (We need the poly-time algorithm $\textbf{Aug}_{\mathcal{A}}$ that does the augmentation with an additional cost of at most $\beta\xi(S', j)$.\\
\end{defi}
Now we will state the \texttt{Ind-Boost} algorithm:\\

\textbf{Algorithm}  \texttt{Ind-Boost}
\begin{enumerate}
	\item Choose a set $D$ by picking each element $j \in V$ with probability $\sigma \pi_j$ independently
	\item Using the algorithm $\mathcal{A}$, construct an $\alpha$-approximation solution $F_0 \in \textbf{Sols}(D)$
	\item Let $S$ be the set of clients realized in the second stage. For each client $j \in S$, use the agumentation algorithm $\textbf{Aug}_{\mathcal{A}}$ to compute $F_j$ so that $F_0 \cup F_j \in  \textbf{Sols}(D \cup \set{j})$. Output $F_S = \bigcup_{j \in S} F_j$ as the second stage solution. The combined solution is $F_0 \cup F_S$.
\end{enumerate}
\vspace{2mm}

\begin{thm}\label{thm2}
Consider a deterministic combinatorial optimization problem $Det$ that is subadditive, and let $\mathcal{A}$ be an $\alpha$-approximation algorithm for $Det$ with a $\beta$-uni-strict cost sharing function. Then  \texttt{Ind-Boost} is an $(\alpha + \beta)$-approximation algorithm for the corresponding two-stage stochastic version of the combinatorial optimization problem
\end{thm}
The proof of this theorem is similar to \cref{thm1}
\vspace{2mm}

\begin{rmk}
These algorithms require the existence of special cost-sharing functions defined above. \cite{gupta_pal_ravi_sinha_2004} shows the existence of such functions for stochastic steiner trees, facility location and vertex cover problems. The constants $\beta$ (of the cost-sharing functions) for these problems are explicit constants and are reasonably small.
\end{rmk}
\vspace{2mm}

\subsection{Sample Average Approximation method (SAA)}
\subsubsection{Problem set-up and algorithm}
The main idea of the algorithm $SAA$ is simple: we sample $\mathcal{N}$ times from distribution on scenarios, and estimate the actual distribution by distribution induced by these samples. Then we solve the deterministic problem for each scenario to get a set of the (approximated) solutions. We then output the weighted average of those solutions based on the estimated distribution. The idea of $SAA$ is similar and a bit simpler than texttt{Boost-and-Sample}, but the details rely a lot on the linear programming formulation of the two-stage stochastic combinatorial problems that we have discussed so far.\\

From what we discuss in section \textbf{1.4}, we assume we have to solve the following convex-relaxation problem:
$$\min_{x} h(x) = w^I\cdot x + \sum_{A \in \mathcal{A}} p_A f_A(x) \text{ on some } \mathcal{P} \subset \mathbb{R}^m_{\geq 0}$$
\begin{align*}
\text{with } f_A(x) = \min_{r_A, s_A} \{w^A\cdot r_A + q^A\cdot s_A&: r_A \in \mathbb{R}^m_{\geq 0}, s_A \in \mathbb{R}^n_{\geq 0},\\
& D^A s_A + T^A r_A \geq j^A - T^Ax \}
\end{align*}

For this algorithm, we further assume that
\begin{itemize}
	\item $T^A \geq 0$ for every scenario $A$,
	\item For every $x \in \mathcal{P}$, $\sum_{A \in \mathcal{A}} p_A f_A(x) \geq 0$, and the primal and dual problems corresponding to $f_A(x)$ are feasible for every scenario $A$. We denote the $z_A^*$ be the optimal dual for the problem with optimal value $f_A(x)$. 
	\item $\mathcal{P} \subset B(0, R)$ where $\ln(R)$ is polynomially bounded (in term of $V$)
\end{itemize}
\vspace{2mm}

We also define $\lambda = \max(1, \max_{A \in \mathcal{A}} \frac{w_S^A}{w_S^I})$, and assume that $\lambda$ is known, and $\mathcal{I}$ be the size of the input (the minimum number of bits to describe the programming problem). We have the following \texttt{SSA} algorithm:\\

\textbf{Algorithm} \texttt{SSA}
\begin{enumerate}
	\item Choose the sample size $\mathcal{N}= O(m\lambda^2\log^2(\frac{2KR}{\epsilon})\log(\frac{2KRm}{\epsilon\delta}))$ ($K$ is the bound for Lipschitz constant of $f_A$ for every scenario $A$)
	\item Sample the subset of clients $A$ from the distribution $\pi$. Then assume that $\mathcal{N}_A$ is the number of times the scenario $A$ is sampled, and let $\hat{p}_A = \frac{\mathcal{N}_A}{N}$
	\item In the final step, we minimize the sample average function $\hat{h}(x) = w^I \cdot x + \sum_{A \in \mathcal{A}} \hat{p}_A f_A(x)$ over $x \in \mathcal{P}$
\end{enumerate}
\vspace{2mm}

\subsubsection{Algorithm's performance guarantee}
Let $\mathcal{P}$ be the bounded feasible region and $R$ be a radius such that $\mathcal{P}$ is contained in the ball $B(0, R) = \set{x: \norm{x} \leq R}$. \\

The idea of the proof is to bound the difference between the sample output $\hat{h}$ and the actual optimal value $h$ by comparing their subgradients or near subgradients at points on $\mathcal{P}$ (because one vector can be subgradient for $1$ function but ca only be near-subgradient for the other one). Near-subgradient is formalized in terms of $\omega$-subgradient. Interested readers can see \cite{shmoys_swamy} to understand this idea from ellipsoid algorithm's point of view. Moreover, we don't compare these subgradients at every point in $\mathcal{P}$, but only on a discrete (extended) grid $G$ of $\mathcal{P}$\\

The subgradients are still hard to study, but we can derive a nice formula that expresses the subgradient in terms of certain weighted sums of the dual solution $z_A^*$ we mentioned above. The sum is taken over all scenarios $A$. This formula can be used to characterize the subgradients of both $h$ and $\hat{h}$. From here, usual probabilistic estimates (Chebyshev-like inequality) can be used to prove the approximation guarantee of the algorithm with high probability. We will not focus on the detailed proof of the guarantee theorem, so we will occasionally skip proofs of the lemmas used in the main proof. We start with the definition of extended grid, $\omega$-subgradient and subgradient similarity property of two functions \cref{propA}. Then we go on to list several lemmas and provide the main proof for guaranteed theorem.
\vspace{2mm}

\begin{defi}
Let $\epsilon, \gamma > 0$ be two parameters with $\gamma \leq 1$. Set $N = \log(\frac{2KR}{\epsilon})$ and $\omega = \frac{\gamma}{8N}$. Define $$G' = \set{x \in \mathcal{P}: x_i = n_i(\dfrac{\epsilon}{KN\sqrt{m}}), n_i \in \mathbb{Z} \text{ for all } i= 1, \cdots, m}$$
Then set $$G = G' \cup \set{x+t(y-x), y+t(x-y): x, y \in G', t = 2^{-i}, i \in \overline{1, N}}$$. 

We call $G$ the \textbf{extended $\dfrac{\epsilon}{KN\sqrt{m}}$-grid} of the polytope $\mathcal{P}$. We easily see that for every $x \in \mathcal{P}$, there exists $x' \in G'$ such that $\norm{x-x'} \leq \frac{\epsilon}{KN}$. 
\end{defi}
\vspace{2mm}

\begin{lem}\label{lem1}
Let $G$ be the extended $\epsilon$-grid of $\mathcal{P}$. Then $|G| \leq N(\frac{2R}{\epsilon})^2m$.
\end{lem}
\vspace{2mm}

Let $\norm{u}$ denote the $l_2$ norm of $u$. We say that function $g: \mathbb{R}^m \to \mathbb{R}$ is Lipschitz with Lipschitz constant $K$ iff $|g(u)-g(v)| \leq K \norm{u-v}$ for all $u, v \in \mathbb{R}^m$. Let $g: \mathbb{R}^m \to \mathbb{R}$ and $\hat{g}: \mathbb{R}^m \to \mathbb{R}$ be two functions with Lipschitz constant $K$.

\begin{defi}
We say that $d$ is a subgradient of a function $g: \mathbb{R}^m \to \mathbb{R}$ at the point $u$ if the inequality $g(v)-g(u) \geq d\cdot(v-u)$ holds for every $v \in \mathbb{R}^m$. We say that $\hat{d}$ is an $\omega$-subgradient of $g$ at $u \in \mathcal{D}$ if for every $v \in \mathcal{D}$, we have $g(v)-g(u) \geq \hat{d}\cdot(v-u) - \omega g(u)$.
\end{defi}
\vspace{2mm}

\begin{defi}
We say that $g$ and $\hat{g}$ with Lipschitz constant $K$ satisfy property \cref{propA} iff
\begin{equation}\label{propA}
\begin{split}
\forall x \in G, \exists\ \hat{d_x}\in \mathbb{R}^m: \hat{d_x}\text{ is a subgradient of } \hat{g} \text{ and an } \omega-\text{subgradient of } g \text{ at }x
\end{split} \tag{A}
\end{equation}
\end{defi}
\vspace{2mm}

Now we will state a lemma concerning the similarity of the optimal minimum value of two functions $g$ and $\hat{g}$ if they both satisfy property \cref{propA}. The proof of this lemma mainly includes algebra and analysis estimates and we will not go into that.
\begin{lem}\label{lem2}
Recall that we have $N = \log(\frac{2KR}{\epsilon})$ and $\omega = \frac{\lambda}{8N}$. Suppose $g$ and $\hat{g}$ satisfy property \cref{propA}. Let $x^{*}, \hat{x} \in \mathcal{P}$ be points that respectively minimize $g$ and $\hat{g}$, with $g(x^{*}) \geq 0$. Then
$$g(\hat{x}) \leq (1+\gamma)g(x^*) + 6\epsilon$$
\end{lem}
\vspace{2mm}

\begin{lem}\label{lem3}
Let $d$ be a subgradient of $h$ at the point $x \in \mathcal{P}$, and suppose $\hat{d}$ is a vector such that $d - \omega w^I \leq \hat{d} \leq d$ for all $e \in X$. Then $\hat{d}$ is an $w$-subgradient of $h$ at $x$
\end{lem}
\begin{proof}
Let $y \in \mathcal{P}$. We have $h(y) - h(x) \geq d\cdot (y-x) = \hat{d}\cdot(y-x) + (d - \hat{d})\cdot(y-x)$. Since $0 \leq d_e - \hat{d}_e \leq \omega w_e^I$ and $x_e, y_e \geq 0$, we have, $(d - \hat{d})\cdot(y-x) \geq -(d-\hat{d})x \geq -\sum_e \omega w_e^I x_e \geq -\omega h(x)$
\end{proof}
\vspace{2mm}

\begin{lem}\label{lem4}
Let $x \in \mathbb{R}^m$, and let $z_A^*$ be an optimal dual solution for scenario $A$ with $x$ as the stage $I$ vector. The vector $d = =w^I - \sum_A p_A(T^A)^Tz_A^*$ is a subgradient of $h$ at $x$ and $\norm{d} \geq \lambda\norm{w^I}$
\end{lem}
\vspace{2mm}

\begin{rmk}
Because the sample average function $\hat{h}$ has the same form as $h$ (but with a different distribution), we can also show by a similar proof that $\hat{d_x} = w^I - \sum_A \hat{p}_A(T^A)^Tz_A^*$ is also a subgradient of $\hat{h}$ at $x$ and $\norm{\hat{d_x}} \leq \lambda \norm{w^I}$. From the bound on $d_x$ and $\hat{d_x}$, one can show that the Lipschitz constant of $h$ and $\hat{h}$ is at most $K = \lambda\norm{w^I}$. Moreover, since $\hat{d_x}$ is the average over the random scenario samples, $\E[\hat{d_x}] = d_x$. This is exactly like if we have i.i.d sample $X_1, \cdots, X_k$ of some random variable $X$, $\E[(\sum X_i)/k] = \E[X]$ ($X_i$ is also a random variable with the same distribution as $X$)
\end{rmk}
\vspace{2mm}

Finally, we state a lemma that helps us upper-bound a random variable based on its independent samples.
\begin{lem}\label{lem5}
Let $X_i$ for $i \in 1, \cdots, N$ be i.i.d random variable where each $X_i \in [-a, b], a, b > 0, \alpha = max(1, a/b)$, and $c$ is an arbitrary positive number. Let $X = \dfrac{\sum_{i=1}^N X_i}{N}$ and $\mu = \E[X] = \E[X_i]$. Then:
$$\Prb(X \in [\mu-bc, \mu+bc]) \geq (1-\delta)$$
\end{lem}
\vspace{2mm}

Now we come to the main theorem for performance guarantee of \texttt{SSA}:
\begin{thm}\label{thm3}
For any $\epsilon, \lambda > 0\ (\lambda < 1)$ with probability at least $1 - \delta$, any optimal solution $\hat{x}$ to the sample average problem constructed with $poly(\mathcal{I}, \lambda, \frac{1}{\gamma}, \ln(\frac{1}{\epsilon}), \ln(\frac{1}{\delta}))$ (polynomial function in term of those variables in the parentheses) samples satisfies $h(\hat{x}) \leq (1 + \gamma)OPT + 6\epsilon$. Here $OPT$ is the optimum value of $h$ described above.
\end{thm}
\begin{proof}
We show that $h$ and $\hat{h}$ satisfy the property \cref{propA} with probability $1 - \delta$ with the stated sample size, and then the proof is completed by \cref{lem2}. Define $N = \log(\frac{2KR}{\epsilon})$, and $\omega = \frac{\lambda}{8N}$, and let $G$ be the extended $\frac{\epsilon}{KN\sqrt{m}}$ grid. Note that $\log(KR)$ is polynomially bounded in the input size. Let $n = |G|$. Using \cref{lem5}, if we sample $\mathcal{N} = \frac{4(1+\lambda)^2}{3\omega^2} \ln(\frac{2mn}{\delta})$ times to construct $\hat{h}$, then at a given point $x$, subgradient $\hat{d_x}$ of $\hat{h}$ is component-wise close to its expectation with probability at least $1-\frac{\delta}{n}$, so by \cref{lem3}, $\hat{d_x}$ is an $\omega$-subgradient of $h$ at $x$ with high probability. So with probability at least $1-\delta$, $\hat{d_x}$ is an $\omega$-subgradient of $h$ at every point $x \in G$. Using \cref{lem1} to bound $n$, we get that $\mathcal{N}  = O(m\lambda^2\log^2(\frac{2KR}{\epsilon})\log(\frac{2KRm}{\epsilon\delta}))$
\end{proof}
\vspace{2mm}

\section{Correlation Robust Stochastic Optimization}
\subsection{Correlation robust stochastic problem and correlation gap}
From section 1, we see that the two-stage stochastic combinatorial optimization problem can be represented in the compact form:
$$\min_{x \in C} \E[f(x, S)]$$
where $x$ is the decision variable in stage $1$. Note that $x$ is a vector indexed by element $e \in X$, and is the indicator of elements $F_0$ chosen in this stage that we mentioned in section $1$. Now the random set $S$ can only be realized in stage 2, and when it is realized, $f(x, S)$ is the minimum cost function of all elements $\in X$ needed to satisfy the realized set $S$ of clients. In stage 2, elements indicated by the decision variable will have cheaper cost (stage 1 cost).\\

Now if get $\mathcal{L}$ to be the set of all distribution $\mathcal{D}$ on the random subset of $V$, then we can rewrite our stochastic problem
as finding $\min_{x \in C} g(x)$ when $g(x)$ is the worst expectation cost over all distribution $\mathcal{D} \in \mathcal{L}$. In other word, $g(x) = \max_{\mathcal{D} \in \mathcal{L}} \E_{\mathcal{D}}[f(x, S)]$. \\

Now we further restrict $\mathcal{L}$ to the set of all distribution with fixed marginal distribution: $\Prb_{\mathcal{D}}(S: i \in S) = p_i$ for all $i \in V$ ($p_i$ fixed). We denote the restriction $\mathcal{L}$ by $\mathcal{L}_{\set{p_i}}$. Thus, $g(x)$ can be defined as $\max_{\mathcal{D}} \E_{\mathcal{D}}[f(x, S)]$ so that $\sum_{S: i \in S} \Prb_{\mathcal{D}}(S) = p_i\ \forall i \in V$. 

\begin{defi}
We keep $x$ fixed, and call a tuple $(V, f, \set{p_i})$ an instance. and let $\mathcal{D}^R$ to be the  distribution $\mathcal{D}$ in $\mathcal{L}_{\set{p_i}}$ that maximize the expectation of function $f$ above. We also define the distribution $\mathcal{D}^I$ with respect to the tuple $(V, f, \set{p_i})$ simply to be the independent Bernoulli distribution $\mathcal{D}^I(S) = \prod_{i \in S} p_i$. 
\end{defi}

\begin{defi}
Now we define the correlation gap $\kappa$ for the problem instance $(V, f, \set{p_i})$ as the ratio:
$$\kappa = \frac{\E_{\mathcal{D}}^R[f(x, S)]}{\E_{\mathcal{D}^I}[f(x, S)]}$$ 
\end{defi}

The idea is that if we can bound the correlation gap $\kappa$ by some number independent of $x$ and $\set{p_i}$, then we can go over each tuple $\set{p_i}$ and $x$, so that at the end we can bound optimal value with respect to the worst-case distribution by an optimal value with respect to some independent distribution. Thus, we can use the independent distribution for the sampling methods we describe in section $2$ to get a good approximation of the original two-stage stochastic combinatorial optimization problem. Then our goal is accomplished. Now in section 2, we will bound $\kappa$ with the assumption that the function $f$ and $V$ in our problem has some specific form.

\subsection{The bound on correlation gaps of specific problem}
We will start with definition of a $(\eta, \beta)$-cost-sharing scheme.
\begin{defi}
A cost-sharing scheme can be defined as a function $\chi$ so that for each subset $S$ of $V$ with an ordering $\sigma_S$ on $S$, and each $i \in S$, $\chi$ maps the tuple $(i, S, \sigma_S)$ to the non-negative number $\chi(i, S, \sigma_S)$. For a fixed $x$, consider the cost function $f(x, S) = f(S)$ as a function on $2^|V|$. We associate with $f$ a cost-sharing scheme, and call this scheme the cost-sharing scheme for allocating the cost $f$. The cost sharing scheme basically specify the share of $i$ in $S$. Now we define the following three properties of a cost sharing scheme for allocating the cost $f$:  \textbf{$\bm{\beta}$- budget balance},  \textbf{cross-monotonicity}, and \textbf{weak $\bm{\eta}$-summability}
\begin{enumerate}
\item \textbf{$\bm{\beta}$-budget balance}: For all $S$, and ordering $\sigma_S$ on $S$:
$$f(S) \geq \sum_{i=1}^{|S|} \chi(i, S, \sigma_S) \geq \frac{f(S)}{\beta}$$

\item\textbf{Cross-monotonicity}: For all $i \in S$, $S\subset T$, and $\sigma_S \subset \sigma_T$:
$$\chi(i, S, \sigma_S) \geq \chi(i, T, \sigma_T)$$
Here $\sigma_S \subset \sigma_T$ means that the ordering $\sigma_S$ is a restriction of the ordering $\sigma_T$ to subset $S$. 

\item \textbf{Weak $\bm{\eta}$-summability}: For all $S$, and orderings $\sigma_S$ on $S$:
$$\sum_{l=1}^{|S|}\chi(i_l, S_l, \sigma_{S_l}) \leq \eta f(S)$$
where $i_l$ is the $l$th element and $S_l$ is the set of the first $l$ memebers of $S$ according to ordering $\sigma_S$. And $\sigma_{S_l}$ is the restriction of $\sigma_S$ on $S_l$.
\end{enumerate}

We call a cost-sharing scheme satisfying the above three properties an $(\eta, \beta)$-cost-sharing scheme.
\end{defi}

\begin{defi}
A function $f(x, S)$ is non-decreasing in $S$ if for every $x$ and every $S \subset T$, $f(x, S) \leq f(x, T)$.
\end{defi}

In this section, we will bound the correlation gap for two-stage stochastic combinatorial optimization problems with a non-decreasing cost function and an $(\eta, \beta)$-cost-sharing scheme. To prove the bound for a general instance $(f, V, \set{p_i})$, we reduce it to a special instance $(f', V', \set{p_i'})$ by applying a \textbf{split} operation to the original instance, and then prove the bound for the special instance instead. Let us start with the definition of this \textbf{split} operation.

\begin{defi}\label{split_opt}
\textbf{Split} operation is a map from one problem instance to a new problem instance. Given a problem instance $(f, V, \set{p_i})$, and integers $\set{n_i \geq 1, i \in V}$, we define a new instance $(f', V', \set{p_j'})$ as follows: split each item $i \in V$ into $n_i$ copies $C_1^i, C_2^i, \cdots, C_{n_i}^i$, and assign a marginal probability of $p'_{C_k^i} = \frac{p_i}{n_i}$ to each copy. Let $V'$ denote the new ground set that contains all the duplicates. Now we define a new cost function $f': 2^{V'} \to \mathbb{R}$ as:
$$f'(S') = f(\Pi(S')),\text{ for all } S' \subset V'$$
where $\Pi(S') \subset V$ is the orignal subset of elements whose duplicates appear in $S'$, i.e. $\Pi(S') = \set{i \in V| C_k^i \in S' \text{for some } k \in \set{1, 2, \cdots, n_i}}$\\

We say that we apply \textbf{split} operation to obtain new instance $(f', V', \set{p_j'})$ from $(f, V, \set{p_i})$ by splitting each item $i \in V$ into some $n_i$ copies.
\end{defi}

Now we state the main result for the upper bound of the correlation gap:
\begin{thm}\label{thm4}
For any instance $(f, V, {p_i})$, if for all feasible $x$, the cost function $f(x, S)$ is non-decreasing in $S$, and has an $(\eta, \beta)$-cost-sharing scheme for elements in $S$, then teh correlation gap is bounded by $\eta \beta \left(\frac{e}{e-1}\right)$
\end{thm}
\begin{proof}
Assume that $\mathcal{D}$ is the worst-case distribution, the distribution that minimizes the expected cost. We only consider subset $U \subseteq V$ with non-zero probability in $\mathcal{D}$. Whenever two such sets $U$ intersect, we repeatly apply the \textbf{split} operation and split one element that belongs to many subsets $U$ into different copies until we get a new instance $(f', V', \set{p'_i})$ such that no such intersections exist.

Then suppose we get the marginal probability $\set{p_i = \frac{K_i}{K}}$ where $K_i, K \in \mathbb{Z}$ and $K$ is common denominator of ration $\set{p_i}$, then we can apply another split operation by dividing $i$th element into $N_i$ copies so that we get the new marginal probability $p_i = 1/K\ \forall i$. We can combine these two split operations into a single split operation. We start by looking at the reassignment of the distribution in \cref{lem8} because this reassignment gives us the new worst-case distribution for the new instance after the split operation. We see that no probability is assigned to any subset $S'$ of $V'$ containing more than one copy of some items. Thus, after the above combined split operation, in the new instance (with new worst-case distribution) the random set $S$ can only be one of the $K$ sets in a particular partition $\set{A_1, \cdots, A_M}$ of $V$ and the probability a random set $S$ is some $A_j$ is exactly the marginal probability $p_i$ for any $i \in A_j$, and this is $1/K$. Hence, $M = K$, and the worst-case distribution on subsets $S$ of $V$ becomes the uniform distribution of the collection of these $K$ disjoint set $\set{A_k}_{k = 1}^K$. We called this worst-case distribution the \textbf{K-partition-type} distribution\\

Recall that $\mathcal{L}$ and $\mathcal{I}$ are expected cost of worst-case distribution and independent Bernoulli distribution respectively. By \cref{lem8}, split operation preserves $\mathcal{L}$, and by lemma \cref{lem9}, the operation only decreases I, so that the correlation gap of the original instance is bounded by the correlation gap of the new transformed instance with worst-case distribution being uniform distribution. By \cref{lem7}, the new cost function for the new instance is also non-decreasing. By \cref{lem10}, there exists a cost-sharing function $\chi'$ for the new instance that is $\beta$-budget balance, cross-monotone, and weak $\eta$ summability for specific subsets $S' \subseteq T' \subseteq V'$ (and $\sigma_{S'}$ and $\sigma_{T'}$). Thus by invoking the \cref{lem6}, we get the correlation gap upper bound for the special case (new transformed instance). Hence, this is also the upper bound for the correlation gap in the original instance.
\end{proof}

\begin{lem}\label{lem6}
For instances $(f, V, \set{p_i})$ such that:
\begin{itemize}
	\item $f(S)$ is non-decreasing
	\item marginal probabilities $p_i$ are all equal to $1/K$ for some integer $K$
	\item the worst case distribution is a $K$-partition type distribution,
\end{itemize}
the correlation gap is bounded as
$$\frac{\mathcal{L}(f, V, \set{1/K})}{\mathcal{I}(f, V, \set{1/K})} \leq \eta \beta \frac{e}{e-1}$$	
\end{lem}
\begin{proof}
Let the optimal $K$-partition corresponding to the worst case distribution is $\set{A_1, \cdots, A_K}$. Then by the definition of worst case distribution, we obviously have $\mathcal{L}(f, V, {1/K}) = \sum_{k=1}^K \frac{1}{K}f(A_k)$. Assume w.l.o.g that $f(A_1) \geq f(A_2) \geq \cdots f(A_K)$. Fix an order $\sigma$ on elements of $V$ such that for all $k$, the elements of $A_k$ come before $A_{k-1}$. For every set $S$, let $\sigma_S$ be the restriction of ordering $\sigma$ on set elements of set $S$. Let $\chi$ be the $(\eta, \beta)$ cost-sharing scheme for function $f$. Then by weak $\eta$-summability of $\chi$:
\begin{equation}\label{IByPhi}
\mathcal{I}(f, V, \set{1/K}) = \E_{S \subseteq V}[f(S)] \geq \frac{1}{\eta} \E_{S \subseteq V}[\sum_{l=1}^|S| \chi(i_l, S_l, \sigma_{S_l})]
\end{equation}
where the expected value is taken over the independent distribution. Also, every probability and expectation in this proof will be from the independent distribution. For the worst case distribution, we already have the (simple) above deterministic sum for $\mathcal{L}(f, V, \set{1/K})$.\\

Denote $\phi(V):= \frac{1}{\eta} \E_{S \subseteq V}[\sum_{l=1}^{|S|} \chi(i_l, S_l, \sigma_{S_l})]$. We will give a lower bound for $V$ recursively, and therfore give the needed lower bound for $\mathcal{I}(f, V, \set{1/K})$. We will show that
\begin{equation}\label{recursivePhi}
\phi(V) \geq (1-p) \phi(V\setminus A_1) + \frac{1}{\beta} f(A_1)
\end{equation}
We will now prove \cref{recursivePhi} by dividing $\chi$-terms in the expectation in $\phi_V$ into parts have intersection with $A_1$ and parts haven't. For each $S \subseteq V$, let $S_1 = S \cap A_1$, and $S_{-1} = S\cap (V\setminus A_1)$. Beccause of the choice of ordering $\sigma$ and $\set{\sigma_S}_S$, elements in $A_1$ come after elements in $V\setminus A_1$ in any ordering $\sigma_S$. Therefore, all elements of $S_{-1}$ comes before elements of $S_1$. Hence the set of first $|S_{-1}|$ elements of $S$ is exactly $S_{-1}$, and for $l \leq |S_{-1}|, S_l \subset S_{-1}$, and $i_l \in S_1$ for $l > S_{-1}$.
\begin{equation}\label{phi_sum}
\phi(V) = \E_S[\sum_{i=1}^{|S_{-1}|} \chi(i_l, S_l, \sigma_{S_l})] + \E_S[\sum_{l =|S_{-1}|+1}^{|S|} \chi(i_l, S_l, \sigma_{S_l})]
\end{equation}

Because $S_l \subset S \cup A_1$, by cross-monotocity of $\chi$, the second term of $\phi(V)$ in \cref{phi_sum} is bounded below by:
$\E_S[\sum_{l = |S_{-1}|+1}^|S| \chi(i_l, S_l, \sigma_{S_l})] \geq \E_S[\sum_{l =|S_{-1}|+1} ^|S| \chi(i_l, S\cup A_1, \sigma_{S \cup A_1})]$
But because $S^{-1}$ and $S_1$ as set-valued random variables are independent, we can further simplify this lower bound (for the second term):
\begin{align*}
&\E_S[\sum_{l =|S_{-1}|+1} ^{|S|} \chi(i_l, S\cup A_1, \sigma_{S \cup A_1})]\\
&= \E_{S_{-1}}\left[\E_{S_1} \left[\sum_{l=|S_{-1}|+1}^{|S|} \chi(i_l, S_{-1} \cup A_1, \sigma_{S_{-1} \cup A_1})|S_{-1}\right] \right]\\
&=  \E_{S_{-1}}\left[p\sum_{i \in A_1} \chi(i, S_{-1} \cup A_1, \sigma_{S_{-1} \cup A_1})\right] \quad (p = 1/K)\\
&= p \E_{S_{-1}}\left[\sum_{i \in A_1} \chi(i, S_{-1} \cup A_1,  \sigma_{S_{-1} \cup A_1})\right]
\end{align*}
Here instead of taking the entire expectation over $S$, we can keep $S_{-1}$ fixed and take expectation over $S_1$, and then take expectation over $S_{-1}$, and this gives the first equality. This rearrangement is possible because the distribution over $S$ is just the product of the distribution over $S_1$ and $S_{-1}$. Also note that $S\cup A_1 = S_{-1} \cup A_1$.\\

Now in the second equality, the only random variable in the inner expectation (wrt $S_{-1}$) is $i_l$, and basically we are summing over $i \in S_1$ in the inner sum. Now given fixed $S^{-1}$, each $i \in A_1$ is in $S_1$ if and only if it belongs to the random set $S$. Thus, $\Prb_{S_{-1}}(i \in S_1) = \Prb(i \in S) = p = 1/K$.  By considering the inner expectation as the weighted sum of $\chi$ over $i \in A_1$ with weights being $\Prb_{S_{-1}}(i \in S_1)$, we obtain the second inequality.\\

Again, by using indepedence and cross monotonicity, we can bound the first term of \cref{phi_sum}:
\begin{align*}
&\E_S[\sum_{l=1}^{|S_{-1}|} \chi(i_l, S_l, \sigma_{S_l})] = \E_{S_{-1}}[\sum_{l=1}^{|S_{-1}|} \chi(i_l, S_l, \sigma_{S_l})] \\
&\geq (1-p)\E_{S_{-1}}[\sum_{l=1}^{|S_{-1}|} \chi(i_l, S_l, \sigma_{S_l})] + p\E_{S_{-1}}[\sum_{l=1}^{|S_{-1}|} \chi(i_l, S_{-1} \cup A_1, \sigma_{S_{-1} \cup A_1})] \\
&= (1-p)\phi(V \setminus A_1) + p \E_{S_{-1}}[\sum_{l=1}^{|S_{-1}|} \chi(i_l, S_{-1} \cup A_1, \sigma_{S_{-1} \cup A_1})]
\end{align*}
The first equality also follows from taking expectation over $S_{-1}$ and then $S_{1}$ and from the fact that the function we take expectation over only depends on $S_{-1}$. \\

Now combining lower bounds for two terms of \cref{phi_sum}, and then using $\beta$-budget balanced property for $S_{-1} \cup A_1$, we get:
\begin{align*}
\phi(V) &\geq (1-p)\phi(V\AA_1) \\
&+ p \E_{S_{-1}}\left[\sum_{l=1}^{|S_{-1}|} \chi(i_l, S_{-1} \cup A_1, \sigma_{S_{-1} \cup A_1})+ \sum_{i \in A_1} \chi(i, S_{-1} \cup A_1,  \sigma_{S_{-1} \cup A_1}) \right]\\
&\geq (1-p)\phi(V\setminus A_1) + \frac{1}{\beta}p \E_[f(S_{-1} \cup A_1)]\\
&\geq (1-p) \phi(V\setminus A_1) + \frac{1}{\beta}p f(A_1)
\end{align*}
The last inequality follows from monotonicity of $f$.\\

From here by recursively expanding the inequality for $A_2, \cdots, A_K$, we get:
$$\phi(V) \geq \frac{1}{\beta}p \sum_{k=1}^K (1-p)^{k-1}f(A_k)$$
The RHS is a weighted sum of $f(A_k)$. Because both $f(A_k)$ and the weights $(1-p)^{k-1}$ is decreasing in $k$, we can replace the weights in the above weighted sum of $f(A_k)$ by the uniform weights $\dfrac{\sum_{k=1}^K(1-p)^{k-1}}{K}$ to get a lower bound:
\begin{align*}
\phi(V) &\geq  \frac{1}{\beta}p \sum_{k=1}^K f(A_k) \frac{\sum_{k=1}^K(1-p)^{k-1}}{K} = \frac{1}{\beta} \sum_{k=1}^K \frac{1-(1-p)^K}{pK}  pf(A_k)\\
&=\frac{1}{\beta} \sum_{k=1}^K \left(1-\left(1-\frac{1}{K}\right)^K\right)  pf(A_k)\geq \frac{1}{\beta}\left(1 - \frac{1}{e}\right)\sum_{k=1}^K pf(A_k) \quad (p = 1/K)
\end{align*}
The last inequality is because the \\

Thus, by \cref{IByPhi},
$$\mathcal{I}(f, V, \set{1/K}) \geq \frac{1}{\eta \beta}\left(1-\frac{1}{e}\right)(\sum_{k=1}^K pf(A_k)) = \frac{1}{\eta \beta}\left(1-\frac{1}{e}\right) \mathcal{L}(f, V, {1/K})$$
\end{proof}

Now we will state a few properties of \textbf{split} operations in the following four lemmas. We denote $(f', V', \set{p_j'})$ as the ouput of the \textbf{split} operation applying on the instance $(f, V, \set{p_i})$ in all of these lemmas.
\begin{lem}\label{lem7}
If $f(S)$ is a non-decreasing function, then so is $f'$
\end{lem}
\begin{proof}
For any $S' \subseteq T' \subseteq V'$, $\Pi(S') \subseteq \Pi(T')$ because if $\Pi(S')$ has some new element, $S'$ will contain a copy of an element that is not the same as any original element in $V$ of copies in $T'$. Then $S'$ cannot be a subset of $T'$. Contradiction. Thus, by definition of $f'$ and monotonicity of $f$, we have:
$$f'(S') = f(\Pi(S')) \leq f(\Pi(T')) = f'(T')$$
\end{proof}

\begin{lem}\label{lem8}
If the cost function $f$ is non-decreasing in $S$, then the splitting procedure does not change the worst-case expected value:
$$\mathcal{L}(f, V, \set{p_i}) = \mathcal{L}(f', V', \set{p_j'})$$
\end{lem}
\begin{proof}
Fix an $x$, and write $f(S) = f(x, S)$. We represent each distribution over subsets $S$ of set $V$ by the tuple $\set{\alpha_S}_{S \subseteq V}$. The worst-case expected cost is then the optimal value of the following linear program, where 
\begin{align*}
\mathcal{L}(f, V, \set{p_i}) = \max_{\alpha} &\sum_S \alpha_S f(x, S)\\
s.t &\sum_{S: i \in S} \alpha_S = p_i, \forall i \in V\\
&\sum_S \alpha_S = 1\\
& \alpha_S \geq 0, \forall S \subseteq V
\end{align*}
The first constraint is the marginal probability constraint, and the last two are constraints to make sure that  $\set{\alpha_S}_{S \subseteq V}$ is a probability distribution. The worst case distribution is then the tuple $\set{\alpha_S}_{S \subseteq V}$ that attains the optimal value in the above linear programming problem.\\

Every split operation is, in fact, just the composition of a sequence of split operation that produce copies on a single element of $V$. Thus, w.l.o.g, we can assume that our operation only splits the first item into $n_1$ pieces, and keeps everything else the same. Let ${\alpha_S}$ denote the optimal solution (in the linear programming problem) for the instance $(f, V, \set{p_i})$: $\mathcal{L}(f, V, \set{p_i}) =\sum_S \alpha_S f(S)$. Now we can construct a solution $\set{\alpha'_{S'}}_{S' \subset V'}$ for the corresponding linear program of the new instance $(f, V',\set{p_j'})$ so that its objective value $\sum_{S' \subseteq V'} \alpha_{S'}f(S')$ is the same as $\mathcal{L}(f, V, \set{p_i})$. For all $S' \subseteq V'$, let
$$\alpha'_{S'} = \begin{cases}
\alpha_{S'}, & \text{ if } S' \text{ contains no copies of item } 1\\
\frac{1}{n_1}\alpha_{S'}, & \text{ if } S' \text{ contains exactly one copy of item } 1\\ 
0, & \text{ otherwise}
\end{cases}$$
Note that when $S'$ contains at most $1$ copy of item $1$, then $S'$ can be considered as a subset of $V$ by projecting the only copy $C^1_j$ of item $1$ of $S'$ to $1 \in V$, and so $\alpha_{S'}$ is well-defined in the above definition. We will prove that $\alpha'$ is a feasible solution in its corresponding linear program (the same one as above but we replace $p_1$ by $p_1/n_1$ for each duplicate $C^1_j$ with $j \in \overline{1, n_1}$)\\

For set $S$ not containing $1$, $\alpha$ and $\alpha'$ are the same, and we don't need to care about the $S'$ with more than $2$ duplicates of $1$ because $\alpha'_{S'} = 0$. Now for every $S \subset V$ containing $1$, we have only $\alpha_S$ in the original distribution. But in the distribution on $2^{V'}$, we now have $n_1$ different $\alpha_{S'}$ that are all correspond to the same $\alpha_S$ with $S' = (S\setminus \set{1}) \cup \set{C^1_j}$ for $j \in \overline{1, n_1}$. However, they are all the same and equal to $1/n_1$ of $\alpha_S$. So when we do the sum over $\alpha_S$ or do the sum with weight $f(x, S) = f'(x, S')$ (by the definition of $f'$), the sum stays the same when we go from $\alpha$ to $\alpha'$. Thus, the values of the objective functions of $\alpha$ and $\alpha'$ are the same, and the second constraint is satisfied for the linear program on $\alpha'$. Moreover, the first constraint for linear program on $\alpha'$ is also satisfied because:
$$\sum_{S': C^1_j \in S'}\alpha_S' = \sum_{S:1 \in S}\frac{p_1}{n_1} \alpha_S = \frac{1}{n_1} p_1 = \frac{p_1}{n_1}\ \forall j \in \overline{1, n_1}$$
Thus, $\mathcal{L}(f, V, \set{p_i}) = \sum_{\alpha_S'}\alpha'_{S'}f'(S') \leq \mathcal{L}(f', V', \set{p_j'})$.\\

Now we prove the reverse direction. Consider the optimal solution $\set{\alpha'_{S'}}$ of the new instance. There exists the optimal solution that $\alpha'_{S'} = 0$ for all $S'$ that contain more than one copy of item $1$. To see this, assume for contradiction that some set $S$ with non-zero probability has at least two copies of item $1$. By definition of $f'$, removing one copy will not decrease the function value. So we will move one copy $C^1_j$ of $1$ in this set $S$ to a set $T$ with no copy of $1$. In other word, we set $\alpha_S$ to $0$ and add $\alpha_S$ to $\alpha_{S \setminus \set{C^1_j}}$, and do a similar thing for $T$ and $T \cup \set{C^1_j}$. This moving procedure makes sure that the marginal probability constraint is satisfied, and the new objective value is not decreased. The probability distribution constraint on $\alpha'$ is also satisfied. Therefore, we have a new optimal solution $\alpha'$ so that the total number of duplicates (of $1$) on all sets $S'$, which have $\alpha'_{S'} \neq 0$ and contain more than one duplicate, is decreased. Keep doing this will help us obtain the optimal solution $\alpha'$ such that $\alpha'$ is only non-zero on the sets with at most $1$ duplicate. From here, we can reverse the map we defined before to get a feasible solution $\set{\alpha_S}_{S \subseteq V}$ on the linear program of $(f, V, \set{p_i})$. In fact, we can define $\alpha_S: = \sum_{S': \Pi(S') = S} \alpha'_{S'}$. Then, by a similar argument as in the $\leq$ direction, we obtain the $\geq$ direction.
\end{proof}

\begin{lem}\label{lem9}
If $f$ is non decreasing, then after splitting
$$\mathcal{I}(f', V', {p_j'}) \leq \mathcal{I}(f, V, \set{p_i})$$
\end{lem}
\begin{proof}
First all probabilities and expectations are taken under independent Bernoulli distribution (over the subsets of $V$ and $V'$). Again, w.l.o.g, we can assume that $(f', V', \set{p_j'})$ is the new instance by splitting item $1$ into $n_1$ pieces because the general split operation is just the composition of the split operation on a single item. Denote
$$\Lambda := \set{S' \subset V'| S' \text{ contains at least one copy of 1}},$$
and denote $\pi = \Prb(S' \in \Lambda)$. Then
\begin{align*}
\pi &= 1- \Prb(S' \text{ doesn't contain } C^1_j\ \forall j \in \overline{1, n_1}) = 1 -\prod_{i=1}^{n_1}\Prb(C^1_j \not \in S')\\
& = 1 -\prod_{i=1}^{n_1}\left(1 - \frac{p_1}{n_1}\right)= 1 - \left(1-\frac{p_1}{n_1}\right)^{n_1}\leq p_1
\end{align*}

Consider the expected cost under independent Bernoulli distribution, we have:
\begin{align*}
\mathcal{I}(f', V', \set{p_j'}) &= \E_{S'}[f'(S')1_{S' \in \Lambda}] + \E_{S'}[f'(S')1_{S' \not\in \Lambda}]\\
&= \pi\E_{S \subseteq V \setminus \set{1}}[f(S\cup \set{1})] + (1-\pi)\E_{S \subseteq V\setminus\set{1} }[f(S)]\\
&\leq \E_{S \subseteq V \setminus \set{1}}[f(S\cup \set{1})] + (1-\pi)\E_{S \subseteq V\setminus\set{1} }[f(S)] = \mathcal{I}(f, V, {p_i})
\end{align*}
The last equality follows from the fact that under independent distribution, the distribution over subsets of $V$ is just the product of the marginal (Bernoulli) distribution of item $1$ with probability $p_1$ and the distribution over subsets of $V \setminus \set{1}$. The second equality also follows from this fact and from the definition of $f'$.
\end{proof}

\begin{lem}\label{lem10}
Given $(\eta, \beta)$ cost-sharing scheme $\chi$ for $(f, V, \set{p_i})$, there exists a cost-sharing scheme $\chi'$ for instance $(f', V', \set{p_i'})$ constructed by splitting in \cref{TBA} such that $\chi'$ is
$\beta$-budget balanced, weak $\eta$-summable. Moreover, suppose that the splitting procedure produces the new instance of the $K$-partition type such that $V'$ is the disjoint union of its subsets $\set{A_k}_{k = 1}^K$, and no $A_k$ contains two copies of the same element in $V$. Then the cost-sharing scheme $\chi'$ constructed from $\chi$ also satisfies cross-monotonicity for any $S' \subset T'$, $\sigma'_{S'}$, and $\sigma'_{T'}$ (and $i' \in S'$) that meet the following two requirements:
\begin{enumerate}
	\item There exist an ordering $\sigma'$ on $V'$ so that elements of $A_{k+1}$ come before $A_k$ with repect to $\sigma'$, and $\sigma_{T'}$ and $\sigma'_{S'}$ are the restricted orderings of $\sigma$ on $S'$ and $T'$. 
	\item $S'$ is a "partial prefix" of $T'$, i.e., $S' \subseteq A_K \cup \cdots \cup A_k$ and $T'\setminus S' \subseteq A_k \subset \cdots \subset A_1$ for some $k \in \overline{1, K}$
\end{enumerate}
\end{lem}
\begin{proof}
Given cost-sharing scheme $\chi$, construct $\chi'$ as follows: cost-share $\chi'$ coincides with the original scheme $\chi$ for the sets without duplicates, but for a set with duplicates , it assign cost-share solely to the copy with smallest index (as per the input ordering). In other word, for $S' \subset V'$, ordering $\sigma'_{S'}$, and item $C^i_j \in S'$, we allocate cost shares as follows:
$$\chi'(C^i_j, S', \sigma'_{S'}) = \begin{cases}
\chi(i, S, \sigma_S), & \quad j = \min\set{h: C^i_h \in S'}\\
0, & \quad \text{ otherwise}
\end{cases}$$
where $S = \Pi(S')$ defined above, $\sigma_S$ is the ordering of loewst index copies in $\sigma'_{S'}$, and $\min$ is taken with respect to the ordering $\sigma'_{S'}$. \\

For $S' \subset V'$, let $S = \Pi(S')$. Consider an ordering $\sigma'_{S'}$, and record all lowest indexed copies in $S'$ of element in $S$ when we tranverse elements of $S'$ based on this ordering. Any two of these copies will be copies of different elements in $S$, and moreover, the order of these lowest indexed copies induces the ordering $\sigma_S$ on $S$. Assume that $i'_{l_j}$ are these lowest-index copies in $S'$ of $i_j \in S$ for $j \in \overline{1, |S|}$. By definition of $\chi'$ and by our choices of $i'_{l_j}$, for all other elements $i'_l \neq i'_{l_j}$ of $S'$, $\chi'(i'_l, S'_l, \sigma_{S'_l}) = 0$. This is because based on the ordering $\sigma_{S'_l}$ on $S'$, which is the induced ordering of $\sigma'_{S'}$ on first $l$ element of $S'$, $S'_l$ already contains at least one copy other than $i'_l$ of some element in $S$. Thus,
$$\sum_{l = 1}^{|S'|} \chi'(i'_l, S'_l, \sigma'_{S'_l}) = \sum_{j = 1}^{|S|} \chi'(i'_{l_j}, S'_{l_j}, \sigma'_{S'_{l_j}}) = \sum_{j = 1}^{|S|}\chi(i_j, S_j, \sigma_{S_j})$$
The second equality follows from the definition of $\chi$. Then because of the weak $\eta$-summability of $\chi$, for any $S' \subset V$, we get:
$$\sum_{l = 1}^{|S'|} \chi'(i'_l, S'_l, \sigma'_{S'_l}) = \sum_{j = 1}^{|S|}\chi(i_j, S_j, \sigma_{S_j}) \leq \eta f(S) = \eta f'(S')$$
The last equality follows from definition of $f'$. Thus, we obtain the weak $\eta$-summability of $\chi'$.\\

By a similar argument with the same choice of lowest-index copies $i'_{l_j}$, for any $S' \subset V'$, and $S = \Pi(S')$, we get that $$\sum_{l = 1}^{|S'|} \chi(i'_l, S', \sigma_{S'}) = \sum_{j = 1}^{|S|} \chi(i'_{l_j}, S', \sigma_{S'}) = \sum_{j=1}^{|S|} \chi(i_j, S, \sigma_{S})$$
Again, by $\beta$-budget balance property of $\chi$ and by the fact that $f'(S') = f(S)$, we obtain the $\beta$-budget balance for $\chi'$ (wrt $f'$).\\

Finally to prove the cross-monotonicity of $\chi'$, consider $S' \subset T', \sigma_{S'} \subset \sigma_{T'}$ such that $S'$ is a "partial prefix" of $T'$, i.e., $S' \subseteq A_K \cup \cdots \cup A_k$, and $T'\setminus S' \subseteq A_k \subset \cdots \subset A_1$ for some $k$. Let $S = \Pi(S')$, $T = \Pi(T')$, and $\sigma_S $ and $\sigma_T$ be the orderings (on $S$ and $T$) of the lowest indexed copies in $S', T'$ repectively. We show that $\sigma_S \subseteq \sigma_T$, i.e., ordering of elements of $S$ in $\sigma_S$ is the same as their ordering in $\sigma_T$. Suppose the lowest indexed copy of $s \in S$ in $S'$ is $s'$. Then by definition, the ordering of $s \in S$ with respect to $\sigma_S$ is determined by the ordering of $s' \in S'$ with respect to $\sigma'$. Now the lowest indexed copy of $s \in S$ in $T'$ is still $s'$ unless we have a lower-indexed copy $s''$ of $s$ in $T' \setminus S'$. But if this happens, then because $S' \subseteq A_K \cup \cdots \cup A_k$, and $T'\setminus S' \subseteq A_k \subset \cdots \subset A_1$, and in $V'$, all elements of $A_{j+1}$ comes before those of $A_j$, $s'$ and $s''$ must be both in $A_k$. However, this is impossible because the splitting procedure makes sure that no two copy of the same element in $V$ appear in the same $A_k$. Thus, the lowest indexed copies of elements in $S$ are unchanged when one goes from $S'$ to $T'$, meaning that the ordering $\sigma_S$ on $S$ is preserved in $\sigma_T$. In other word, $\sigma_S \subseteq \sigma_T$. Now, for any $i' \in S'$, if $i'$ is not a lowest indexed copy of $T'$, then $\chi'(i', T', \sigma_{T'}) = 0$, and cross-monotonicity is automatically satisfied. If $i'$ is one of the lowest indexed copy in $T'$, it must also be a lowest indexed copy in $S'$, because the ordering $\sigma'_{S'}$ on $S'$ is the ordering $\sigma'_{T'}$ on $T'$ restricted to the subset $S' \subseteq T'$. Suppose that $i'$ is the copy of $i \in S$. Then by definition of $\chi'$, $\chi'(i', T', \sigma'_{T'}) = \chi(i, T, \sigma_T)$, and $\chi'(i', S', \sigma'_{S'}) = \chi(i, S, \sigma_S)$. Moreover, because $\sigma_S \subseteq \sigma_T$, the cross-monotonicity of $\chi$ implies $\chi(i, T, \sigma_T) \leq \chi(i, S, \sigma_S)$. Therefore, $\chi'(i', T', \sigma'_{T'}) \leq \chi'(i', S', \sigma'_{S'})$, and the cross-monotonicity of $\chi'$ is proved.
\end{proof}

\begin{rmk}
The concept of correlation gap and its reduction to independent distribution requires the existence of a $(\eta, \beta)$-cost sharing scheme. \cite{agrawal_ding_saberi_ye_2010} indicates that these schemes exists for a few number of problems we discussed before. First, if $f(x, S)$ is non-decreasing and submodular in $S$ then we have the associated $(1, 1)$-cost sharing scheme, so that in this case, the correlation gap is bounded by $\frac{e}{e-1}$. The stochastic \textbf{UFL} problem also have $(\log n, 3)$-cost sharing scheme where $n = |V|$ is the number of clients, and so its correlation gap is bounded by $O(log n)$. Finally, for stochastic Steiner tree, we also have $(\log^2 n, 2)$-cost sharing scheme so that the correlation gap is bounded by $O(log^2 n)$, where $n = |V|$
\end{rmk}

% Bibliography
\nocite{*}
\newpage
\bibliography{citation}{}

\begin{thebibliography}{1}

\bibitem{agrawal_ding_saberi_ye_2010}
Shipra Agrawal, Yichuan Ding, Amin Saberi, and Yinyu Ye.
\newblock Correlation robust stochastic optimization.
\newblock {\em Proceedings of the Twenty-First Annual ACM-SIAM Symposium on Discrete Algorithms}, 2010.

\bibitem{dpo}
Chandrajit Bajaj and Minh Nguyen.
\newblock {DPO}: Differential reinforcement learning with application to optimal configuration search.
\newblock {\em arXiv preprint arXiv:2404.15617}, 2024.

\bibitem{RL_for_CO_2}
Irwan Bello, Hieu Pham, Quoc~V. Le, Mohammad Norouzi, and Samy Bengio.
\newblock Neural combinatorial optimization with reinforcement learning.
\newblock 2017.

\bibitem{gupta_pal_ravi_sinha_2004}
Anupam Gupta, Martin Pal, R.~Ravi, and Amitabh Sinha.
\newblock Boosted sampling.
\newblock {\em Proceedings of the thirty-sixth annual ACM symposium on Theory of computing - STOC '04}, 2004.

\bibitem{RL_for_CO}
Nina Mazyavkina, Sergey Sviridov, Sergei Ivanov, and Evgeny Burnaev.
\newblock Reinforcement learning for combinatorial optimization: A survey.
\newblock {\em Computers \& Operations Research}, 134:105400, 2021.

\bibitem{ravi_sinha_2005}
R.~Ravi and Amitabh Sinha.
\newblock Hedging uncertainty: Approximation algorithms for stochastic optimization problems.
\newblock {\em Mathematical Programming}, 108(1):97–114, 2005.

\bibitem{shmoys_swamy}
D.b. Shmoys and C.~Swamy.
\newblock Stochastic optimization is (almost) as easy as deterministic optimization.
\newblock {\em 45th Annual IEEE Symposium on Foundations of Computer Science}.

\bibitem{swamy_shmoys}
C.~Swamy and D.b. Shmoys.
\newblock Sampling-based approximation algorithms for multi-stage stochastic.
\newblock {\em 46th Annual IEEE Symposium on Foundations of Computer Science (FOCS'05)}.

\end{thebibliography}
\bibliographystyle{plain} 
\end{document}